\documentclass{article}
\usepackage{amsthm,amssymb,amsmath,amscd}
\usepackage{hyperref,enumerate}
\usepackage{tikz}
\usepackage{todonotes}
\usepackage{mathabx}
\hypersetup{
    pdftitle=   {Classes of graphs with no long cycle as a vertex-minor are
      polynomially $\chi$-bounded},
    pdfauthor=  {Ringi Kim, O-joung Kwon, Sang-il Oum, and Vaidy Sivaraman},
    pdfkeywords={vertex-minor,chi-bounded,polynomially chi-bounded,pivot-minor,1-join,split decomposition}
}

\newcommand{\gjoin}{\mathcal G^\&}
\newcommand{\gsub}{\mathcal G^*}

\newcommand\abs[1]{\lvert #1\rvert}
\newcommand\tri{\boxslash}

\newcommand\pivot\wedge
\def\K_#1{{K_{#1}}}
\def\S_#1{\overline{K_{#1}}}
\newtheorem{THM}{Theorem}[section]
\newtheorem*{THMMAIN1}{Theorem~\ref{thm:pathhalf}}
\newtheorem{LEM}[THM]{Lemma}
\newtheorem{COR}[THM]{Corollary}
\newtheorem{PROP}[THM]{Proposition}
\newtheorem{CON}[THM]{Conjecture}
\newtheorem{QUE}[THM]{Question}

\author{Ringi Kim\thanks{Supported by the National Research Foundation of Korea (NRF) grant funded by the Korea government (MSIT) (No. NRF-2018R1C1B6003786).}
  \\
  Department of Mathematical Sciences, KAIST
  \and
  O-joung Kwon\thanks{Supported by the National Research Foundation of Korea (NRF) grant funded by the Ministry of Education (No. NRF-2018R1D1A1B07050294).}\\
  Department of Mathematics, Incheon National University
  \and
  Sang-il Oum\thanks{Supported by IBS-R029-C1 and by the National Research Foundation of Korea (NRF) grant funded by the Korea government (MSIT) (No. NRF-2017R1A2B4005020).}\\
  Discrete Mathematics Group, Institute for Basic Science (IBS)\\
  Department of Mathematical Sciences, KAIST
  \and
  Vaidy Sivaraman\\
  Department of Mathematics, University of Central Florida}
\title{Classes of graphs with no long cycle as a vertex-minor are
  polynomially {$\chi$}-bounded}

\begin{document}
\maketitle
\footnotetext[1]{E-Mail addresses: \texttt{ringikim2@gmail.com}, \texttt{ojoungkwon@gmail.com}, \texttt{sangil@ibs.re.kr}, \texttt{vaidysivaraman@gmail.com}}
\begin{abstract}
  A class $\mathcal G$ of graphs is \emph{$\chi$-bounded}
  if there is a function $f$ such that
  for every graph $G\in \mathcal G$
  and every induced subgraph $H$ of $G$,
  $\chi(H)\le f(\omega(H))$.
  In addition, we say that $\mathcal G$ is \emph{polynomially $\chi$-bounded} if $f$ can be taken as a polynomial function.
  We prove that for every integer $n\ge3$, there exists a polynomial $f$
  such that $\chi(G)\le f(\omega(G))$ for all graphs with no vertex-minor isomorphic to the cycle graph $C_n$.
  To prove this, we show that if $\mathcal G$ is polynomially $\chi$-bounded, then
  so is the closure of $\mathcal G$ under taking the $1$-join operation.
\end{abstract}
\section{Introduction}\label{sec:intro}
A class $\mathcal{G}$ of graphs is said to be \emph{hereditary} if for every $G\in \mathcal{G}$, every graph isomorphic to an induced subgraph of $G$ belongs to $\mathcal{G}$.
  A class $\mathcal G$ of graphs is \emph{$\chi$-bounded}
  if there is a function $f$ such that
  for every graph $G\in \mathcal G$
  and every induced subgraph $H$ of $G$,
  $\chi(H)\le f(\omega(H))$. The function $f$ is called a \emph{$\chi$-bounding function}.
  This concept was first formulated by Gy\'{a}rf\'{a}s~\cite{Gyarfas1987}.  
  In particular, we say that $\mathcal G$ is \emph{polynomially $\chi$-bounded} if $f$ can be taken as a polynomial function.

  Recently, many open problems on $\chi$-boundedness
  have been resolved; see a recent survey by Scott and Seymour~\cite{surveychi}.
  Yet we do not have much information on graph classes that are polynomially $\chi$-bounded.
  For instance,
  Gy\'{a}rf\'as~\cite{Gyarfas1987} showed that
  the class of $P_n$-free graphs is $\chi$-bounded
  but 
  it is still open~\cite{ELMM2013,Schiermeyer2016}
  whether it is polynomially $\chi$-bounded for $n\ge 5$. 
  Regarding polynomially $\chi$-boundedness, Esperet proposed the following question, which remains open.
  \begin{QUE}[Esperet; see \cite{KM2016}]
  Is every $\chi$-bounded class of graphs polynomially $\chi$-bounded?
  \end{QUE}

  Towards answering this question, it is interesting to know some graph operations that preserve the property of polynomial $\chi$-boundedness.
  If we have such graph operations, then we can
  use them to generate polynomially $\chi$-bounded graph classes.

  In this direction, Chudnovsky, Penev, Scott, and Trotignon~\cite{CPST2013} showed that if a hereditary class $\mathcal{C}$ is polynomially $\chi$-bounded, then 
  its closure under disjoint union and substitution is again polynomially $\chi$-bounded. 

  We prove the analog of their result for the $1$-join.
  For graphs $G_1$ and $G_2$ with $|V(G_1)|,|V(G_2)|\ge 3$
  and $V(G_1)\cap V(G_2)=\emptyset$, we say that a graph $G$ is obtained from $G_1$ and $G_2$ by  \emph{$1$-join}
if there are vertices $v_1 \in V(G_1)$ and $v_2\in V(G_2)$ such that 
$G$ is obtained from the disjoint union of $G_1$ and $G_2$ by deleting $v_1$ and $v_2$ and adding all edges between every neighbor of $v_1$ in $G_1$ and every neighbor of $v_2$ in $G_2$.
If so, then we say that $G$ is the \emph{$1$-join} of $(G_1,v_1)$ and $(G_2,v_2)$.
For a class $\mathcal G$, let $\gjoin$ be its closure under disjoint union and $1$-join. Note that if $\mathcal G$ is closed under taking isomorphisms, then so is $\gjoin$.
We will see in Section~\ref{sec:prelim} that $\gjoin$ is hereditary if $\mathcal G$ is hereditary.

\begin{THM}\label{THM:main}
  If $\mathcal{G}$ is a polynomially $\chi$-bounded class of graphs,
  then so is $\gjoin$.
\end{THM}

Dvo\v{r}\'{a}k and Kr\'{a}l~\cite{DK2012} and Kim~\cite{R2011} independently showed that for every hereditary class $\mathcal{G}$ of graphs  that is $\chi$-bounded, 
its closure under taking the $1$-joins is again $\chi$-bounded.
However,  in both papers, 
the $\chi$-bounding function $g$ for the new class  is recursively defined  as $g(n)=O(f(n)g(n-1))$ for a $\chi$-bounding function $f$ for $\mathcal{G}$.
So, $g(n)$ is exponential under their constructions.

We shall see that if $f$ is a polynomial, then $g(n-1)$ in the recurrence relation
can be replaced by some polynomial $f^*$.
This technique allows us to prove Theorem~\ref{THM:main}.

As an application, we investigate the following conjecture of Geelen proposed in 2009.
The definition of vertex-minors will be reviewed in Section~\ref{sec:prelim}.
\begin{CON}[Geelen; see \cite{DK2012}]\label{con:geelen}
  For every graph $H$, the class of graphs with no vertex-minor isomorphic to $H$
  is $\chi$-bounded.
\end{CON}
Conjecture~\ref{con:geelen} is known to be true when $H$ is a wheel graph, shown by Choi, Kwon, Oum, and Wollan~\cite{CKOW2017}.
Motivated by the question of Esperet, we may ask the following.
\begin{QUE}\label{q:polygeelen}
  Is it true that for every graph $H$, the class of graphs with no vertex-minor isomorphic to $H$
  is polynomially $\chi$-bounded?
\end{QUE}
If this holds for $H$, then the class of $H$-vertex-minor free graphs satisfies the \emph{Erd\H{o}s-Hajnal property}, which means that there is
a constant $c>0$ such that every graph $G$ in this class has an independent set or a clique of size at least $\abs{V(G)}^c$. Recently, Chudnovsky and Oum~\cite{CO2018} proved that
the Erd\H{o}s-Hajnal property holds for the class of $H$-vertex-minor free graphs for all $H$.

We write $P_n$ to denote the path graph on $n$ vertices
and $C_n$ to denote the cycle graph on $n$ vertices.
Let $K_n\tri\S_n$ be the graph
on $2n$ vertices $\{a_1,a_2,\ldots,a_n,b_1,b_2,\ldots,b_n\}$
such that
$\{a_1,a_2,\ldots,a_n\}$ is a clique,
$\{b_1,b_2,\ldots,b_n\}$ is a stable set,
and 
for all $1\le i,j\le n$,
$a_i$ is adjacent to $b_j$
if and only if $i\ge j$.
See Figure~\ref{fig:ks} for an illustration of $K_6\tri\S_6$.
\begin{figure}
  \centering
  \tikzstyle{v}=[circle, draw, solid, fill=black, inner sep=0pt, minimum width=3pt]
  \begin{tikzpicture}[xscale=1.5]
    \foreach \i in {1,...,6}{
      \node [v,label=right:$a_{\i}$](v\i) at (\i,0) {};
      \node [v,label=$b_{\i}$] (w\i) at (\i,1) {};
    }
    \foreach \i in {1,...,6}{
      \foreach \j in {1,...,\i} {
        \draw (v\i)--(w\j);
        }
    }
    \foreach \i in {2,...,6}{
      \foreach \j in {1,...,\i} {
        \draw  (v\i) to [bend left]  (v\j);
      }
      }
  \end{tikzpicture}
  \caption{$K_6\tri \S_6$.}
  \label{fig:ks}
\end{figure}

In Section~\ref{sec:path}, we prove the following theorem.
\begin{THM}\label{thm:pathhalf}
  Let $n\ge 4$.
  If a graph $G$ has no induced subgraph isomorphic to 
  $P_n$ or $K_{\lceil n/2\rceil}\tri \S_{\lceil n/2\rceil }$,
  then 
 \[\chi(G)\le (n-3)^{\lceil n/2\rceil-1} \omega(G)^{\lceil n/2\rceil-1}.\]
\end{THM}

We use this graph $K_{\lceil n/2\rceil}\tri\S_{\lceil n/2\rceil}$ because
it has a vertex-minor isomorphic to $P_{n}$, shown by Kwon and Oum~\cite[Lemma 2.8]{KO2014}.
\begin{LEM}[Kwon and Oum~{\cite[Lemma 2.8]{KO2014}}]\label{lem:ko}
  The graph $K_{\lceil n/2\rceil}\tri \S_{\lceil n/2\rceil}$ has a vertex-minor isomorphic to $P_{n}$.
\end{LEM}
This allows us to obtain the following corollary of Theorem~\ref{thm:pathhalf}.
\begin{COR}\label{cor:path}
  The class of graphs with no vertex-minor isomorphic to $P_n$
  is polynomially $\chi$-bounded.
\end{COR}
Kwon and Oum~\cite{KO2014} proved the following theorem,
stating that a prime graph with a long induced path must contain a long induced cycle as a vertex-minor.
A graph is \emph{prime} if it is not the $1$-join
of $(G_1,v_1)$ and $(G_2,v_2)$ for some graphs $G_1$, $G_2$ with $\abs{V(G_1)}, \abs{V(G_2)}\ge 3$.

\begin{THM}[Kwon and Oum~\cite{KO2014}]\label{THM:pathtocycle}
  If a prime graph has an induced path of length $\lceil 6.75n^7\rceil$,
  then it has a cycle of length $n$ as a vertex-minor.
\end{THM}

We deduce the following stronger theorem from Corollary~\ref{cor:path}
by using Theorems~\ref{THM:main} and \ref{THM:pathtocycle}.
This answers Question~\ref{q:polygeelen} for a long cycle.

\begin{THM}\label{THM:cycle}
  The class of graphs with no vertex-minor isomorphic to $C_n$
  is polynomially $\chi$-bounded.
\end{THM}
\begin{proof}
  Let $\mathcal G$ be the class of graphs having no
  vertex-minor isomorphic to $P_m$ for $m=\lceil 6.75n^7\rceil$.
  By Corollary~\ref{cor:path},
  $\mathcal G$ is polynomially $\chi$-bounded.
  By Theorem~\ref{THM:main}, $\gjoin$ is polynomially $\chi$-bounded.

  Let $\mathcal H$ be the class of graphs having no vertex-minor
  isomorphic to $C_n$.
  Let $G\in \mathcal H$.
  We claim that $G\in \gjoin$.
  We may assume that $G$ is connected. %
  Every connected prime induced subgraph of $G$ is in $\mathcal G$
  by Theorem~\ref{THM:pathtocycle}.
  Then $G$ can be obtained from
  connected prime induced subgraphs of $G$ %
  by taking $1$-join repeatedly.
  Thus, $G\in \gjoin$. This proves that $\mathcal H\subseteq \gjoin$.
\end{proof}
Because  $C_m$ contains $C_n$ as a vertex-minor whenever $m\ge n$,
we may ask a stronger question on whether or not the class of graphs with no induced subgraph isomorphic to $C_m$ for some $m\ge n$ is polynomially $\chi$-bounded. It is not known.
The following theorem
of  Chudnovsky, Scott, and Seymour~\cite{CSS2015}
was initially a conjecture of Gy\'arf\'as~\cite{Gyarfas1987} in 1985.
\begin{THM}[Chudnovsky, Scott, and Seymour~\cite{CSS2015}]
  The class of graphs with no induced subgraph isomorphic to
  a graph in $\{C_m:m\ge n\}$
  is $\chi$-bounded.
\end{THM}
We remark that as far as we know, it is not known whether the class of graphs with no $P_5$ induced subgraph is polynomially $\chi$-bounded.

This paper is organized as follows.
We will review necessary definitions in Section~\ref{sec:prelim}.
In Section~\ref{sec:proof}, 
we will present a proof of Theorem~\ref{THM:main}.
In Section~\ref{sec:path}, we will prove Theorem~\ref{thm:pathhalf}.
In Section~\ref{sec:p5},
we discuss related problems.

\section{Preliminaries}\label{sec:prelim}

All graphs in this paper are simple and undirected.
For a graph $G$, let $V(G)$ and $E(G)$ denote the vertex set and the edge set of $G$, respectively. 
A \emph{clique} of
a graph is
a set of pairwise adjacent vertices.
For a graph $G$, let $\omega(G)$ be the maximum
number of vertices in a clique of $G$
and $\chi(G)$ be the chromatic number of $G$.

Let $G$ be a graph.
For a vertex subset $S$ of $G$, we denote by $G[S]$ the subgraph of $G$ induced by $S$. 
For a vertex $v$ of $G$, we denote by $G\setminus v$ the graph obtained from $G$ by removing $v$.
For an edge $e$ of $G$, we write $G\setminus e$ to denote the subgraph obtained from $G$ by deleting $e$.
For $v\in V(G)$,
let $N_G(v)$ be the set of neighbors of $v$ in $G$.
For a set $X$ of vertices, let
$N_G(X)=(\bigcup_{v\in X} N_G(v))\setminus X$.

For two graphs $G_1$ and $G_2$, the \emph{disjoint union} of $G_1$ and $G_2$ is a graph $(V(G_1')\cup V(G_2'), E(G_1')\cup E(G_2'))$
where $G_1'$ is an isomorphic copy of $G_1$ and $G_2'$ is an isomorphic copy of $G_2$
such that $V(G_1')\cap V(G_2')=\emptyset$. If $V(G_1)\cap V(G_2)=\emptyset$, then we take $G_1'=G_1$ and $G_2'=G_2$ for convenience.

For two graphs $G_1$ and $G_2$ on disjoint vertex sets and a vertex $v\in V(G_1)$,
we say that a graph $G$ is obtained from $G_1$ by \emph{substituting} $G_2$ for $v$ in $G_1$, 
if  
\begin{itemize}
\item $V(G)=(V(G_1)\setminus \{v\})\cup V(G_2)$,
\item $E(G)=E(G_1\setminus v)
  \cup E(G_2)
  \cup \{ xy: x\in N_{G_1}(v), y\in V(G_2)\}$.
\end{itemize}

For two sets $A$, $A'$ with $A\subseteq A'$, we say that 
a function $f':A'\to B$ \emph{extends}
a function $f:A\to B$
if 
$f'(a)=f(a)$
for all $a\in A$.

\begin{LEM}\label{lem:hereditary}
  If $\mathcal G$ is a hereditary class of graphs, then so is $\gjoin$.
\end{LEM}
\begin{proof}
  We show that for every $G\in \gjoin$ and $v\in V(G)$, 
  $G\setminus v \in \gjoin$.
  We proceed by induction on $|V(G)|$.
  If $G \in \mathcal{G}$, then we are done since $\mathcal{G}$ is hereditary. 
  
  So, we may assume that $G \not\in \mathcal{G}$. 
  Then, $G$ is the disjoint union of $G_1$ and $G_2$ for some $G_1,G_2 \in \gjoin$ 
  or the $1$-join of $(G_1,v_1)$ and $(G_2,v_2)$ for  some $G_1,G_2 \in \gjoin$  
  and $v_i \in V(G_i)$ for $i=1,2$ where $|V(G_1)|, |V(G_2)|\ge 3$.
  
  In the first case, we may assume that $v\in V(G_1)$. 
  Then, by the induction hypothesis, $G_1\setminus v \in \gjoin$, 
  and since $G\setminus v$ is the disjoint union of $G_1\setminus v$ and $G_2$, 
  it follows that $G\setminus v$ is contained in $\gjoin$.
  
  In the second case, by symmetry, we may assume that $v\in V(G_1)\setminus \{v_1\}$.
  By the induction hypothesis, $G_1\setminus v\in\gjoin$.
  If $|V(G_1)|>3$, then $G\setminus v$ is in $\gjoin$ 
   since it is the $1$-join of $(G_1\setminus v, v_1)$ and $(G_2,v_2)$.
   If $\abs{V(G_1)}=3$, then $G\setminus v$ is isomorphic to either $G_2$ or the disjoint union of  $K_1$ and $G_2\setminus v_2$. In either case $G\setminus v$ is contained in $\gjoin$. 
 \end{proof}

 \paragraph{Vertex-minors.}
 Now we are going to define vertex-minors.
 Actually, the readers may skip this part, if they assume Lemma~\ref{lem:ko} and Theorem~\ref{THM:pathtocycle} from other papers.
 
For a vertex $v$ in a graph $G$, the \emph{local complementation} at $v$
results in the graph obtained from $G$ by replacing the subgraph of $G$ induced on $N_G(v)$ 
by its complement.
We write $G*v$ to denote the graph obtained from $G$ by applying
local complementation at $v$.
In other words, $G*v$ is a graph on $V(G)$ such that
two distinct vertices $x$, $y$ are adjacent in $G*v$
if and only if
exactly one of the following holds.
\begin{enumerate}[(i)]
\item Both $x$ and $y$ are neighbors of $v$ in $G$.
\item $x$ is adjacent to $y$  in $G$.
\end{enumerate}
A graph $H$ is  \emph{locally equivalent} to $G$ if $H$ can be obtained from $G$ by a sequence of local complementations.
We say that a graph $H$ is a \emph{vertex-minor} of a graph $G$
if $H$ is an induced subgraph of a graph locally equivalent to $G$.

For an edge $uv$ of a graph $G$, the \emph{pivot} at $uv$
is an operation to obtain $G*v*u*v$ from $G$. We write $G\pivot uv:=G*u*v*u$.
We say that a graph $H$ is a \emph{pivot-minor} of a graph $G$
if $H$ is obtained from $G$ by a sequence of pivots and vertex deletions.

\section{Polynomially $\chi$-boundedness for $1$-join}\label{sec:proof}
For a class $\mathcal G$ of graphs, 
let $\gsub$ be the closure of $\mathcal G$ under disjoint union and substitution. 
We will use the following result due to Chudnovsky, Penev, Scott, and Trotignon~\cite{CPST2013}.
\begin{THM}[Chudnovsky, Penev, Scott, and Trotignon~\cite{CPST2013}]\label{thm:substitution}
  If $\mathcal{G}$ is a polynomially $\chi$-bounded class of graphs, 
  then $\gsub$ is polynomially $\chi$-bounded.
\end{THM}
	The following observation relates two graph classes $\gjoin$ and $\gsub$.
\begin{LEM}\label{LEM:substitution}
Let $\mathcal G$ be a hereditary class of graphs. If $G\in \gjoin$ and $v\in V(G)$, then $G[N_G(v)] \in \gsub$.
\end{LEM}
\begin{proof}
We prove by induction on $\abs{V(G)}$.

If $G\in \mathcal{G}$, then we are done, since $\mathcal{G}$ is closed under induced subgraphs and $\mathcal{G}\subseteq \gsub$.
If $G$ is the disjoint union of two graphs $G_1, G_2$ from $\gjoin$ and $v\in V(G_1)$, 
then by the induction hypothesis on the graph $G_1$, the claim follows.

Suppose that $G$ is the $1$-join of two graphs $(G_1,v_1)$ and $(G_2,v_2)$ where
$G_1,G_2\in\gjoin$ and $\abs{V(G_1)}, \abs{V(G_2)}\ge 3$. 
Without loss of generality, we may assume that $v\in V(G_1\setminus v_1)$.
 Let $G_1[N_{G_1}(v)]=G_v$, and $G_2[N_{G_2}(v_2)]=G_2'$.
 Then, by the induction hypothesis, $G_v\in \gsub$, and $G_2' \in \gsub$ because $\abs{V(G_1)}, \abs{V(G_2)}<\abs{V(G)}$.

 Since $\mathcal G$ is hereditary, so is $\mathcal \gsub$.
 We may assume that $G_2'$ has at least one vertex
 because otherwise $G[N_G(v)]=G_v\setminus v_1\in\mathcal \gsub$.
 We may also assume that $v$ is adjacent to $v_1$ in $G_1$
 because otherwise $G[N_G(v)]=G_v\in\mathcal \gsub$.
 Then $G[N_G(v)]$ can be obtained from $G_v$ by substituting $G_2'$ for $v_1$ and therefore
 $G[N_G(v)]$ belongs to $\gsub$. This completes the proof.
\end{proof}

Let us now define a structure to describe how a connected graph in $\gjoin$
is composed from graphs in $\mathcal G$.
A \emph{composition tree} is a triple $(T,\phi, \psi)$ of a tree $T$, a map $\phi$ defined on $V(T)$ and a map $\psi$ defined on $E(T)$ such that
\begin{itemize}
\item for $t\in V(T)$, $\phi(t)$ is a connected graph, say $G_t$, on at least $3$ vertices where
  graphs in $\{G_t: t\in V(T)\}$ are vertex-disjoint,
\item for $st \in E(T)$, $\psi(st)=\{u,v\}$ for some $u\in V(G_s)$ and $v\in V(G_t)$, and
\item for distinct $e_1 \neq e_2 \in E(T)$, $\psi(e_1)$ and $\psi(e_2)$ are disjoint.
\end{itemize}
If a composition tree $(T,\phi,\psi)$ is given, then one can construct a connected graph $G$ from $(T,\phi,\psi)$ by taking $1$-joins repeatedly as follows: 
\begin{itemize}
\item if $|V(T)|=1$, say $V(T)=\{t\}$, then $G=\phi(t)$.
\item if $|V(T)|>1$, let $e=t_1t_2 \in E(T)$ and $T_i$ be the subtree of $T\setminus e$ containing $t_i$ for each $i=1,2$. Let $\phi_i$ be the restriction of $\phi$ on $V(T_i)$ and $\psi_i$ be the restriction of $\psi$ on $E(T_i)$ for each $i=1,2$. Let $G_i$ be a graph constructed from $(T_i,\phi_i,\psi_i)$ for $i=1,2$. Then, $G$ is the $1$-join of $(G_1,v_1)$ and $(G_2,v_2)$ where $v_i \in V(G_i)\cap \psi(e)$.
  It is straightforward to see that the choice of $e$ does not make any difference
  to the obtained graph $G$.
\end{itemize}
If a vertex $v$ of $\phi(t)$ for some node $t$ of $T$
is in $\psi(e)$ for some edge $e$ of $T$, then $v$ is called a \emph{marker vertex}.
After applying all $1$-joins, marker vertices will disappear.

\begin{LEM}\label{lem:composition}
  Let $\mathcal G$ be a class of graphs.
  Let $G$ be a connected graph in $\gjoin$
  with at least three vertices.
  Then there exists a composition tree $(T,\phi,\psi)$
  that constructs $G$
  such that $\phi(t)\in \mathcal G$ for every node $t$ of $T$.
\end{LEM}
\begin{proof}
  We proceed by induction on $\abs{V(G)}$.
  We may assume that  $G\notin \mathcal G$.
  Since $G$ is connected,
  $G$ is the $1$-join of $(G_1,v_1)$ and $(G_2,v_2)$
  for some graphs $G_1$, $G_2$ in $\gjoin$
  and $v_1\in V(G_1)$, $v_2\in V(G_2)$
  where $\abs{V(G_1)}, \abs{V(G_2)}\ge 3$.
  Since $G$ is connected, both $G_1$ and $G_2$ are connected.
  By the induction hypothesis, we obtain two composition trees.
  We can combine them to obtain a composition tree $(T,\phi,\psi)$ constructing $G$.
\end{proof}

\begin{LEM}\label{LEM:main}
  Let $c_1$, $c_2$ be positive integers.
  Let $G$ be a connected graph constructed by a composition tree
  $(T,\phi,\psi)$ such that
  $\phi(t)$ is $c_1$-colorable for each node $t$ of $T$
  and 
  $G[N_G(w)]$ is $c_2$-colorable
  for each vertex $w$ of $G$.
  Let $v$ be a vertex of $G$.
  Then for every proper $c_2$-coloring $\beta$ of $G[N_G(v)]$, there
  exist functions $\alpha':V(G)\setminus \{v\}\to \{0,1,2,\ldots,c_1\}$
  and
  $\beta':V(G)\setminus\{v\}\to\{1,2,\ldots,c_2\}$ 
  such that
  \begin{enumerate}
  \item [(1)]
    $\alpha'(w)=0$ and $\beta'(w)=\beta(w)$ for every neighbor $w$  of $v$,
  \item [(2)]
  $c=\alpha'\times \beta'$ is a proper $(c_1+1)c_2$-coloring of $G\setminus v$.
  \end{enumerate}

\end{LEM}

\begin{proof}
	We proceed by induction on $|V(G)|$.

	If $\abs{V(T)}=1$, then $G=\phi(t)$ for the unique node $t$ of $T$
        and so $G$ has a proper $c_1$-coloring $h: V(G\setminus v) \to \{1,2,\ldots,c_1\}$.
	We define $\alpha'$ and $\beta'$ on $V(G\setminus v)$ as follows: 
        \[
\alpha'(w)=
\begin{cases}
  0 &\text{if $w\in N_G(v)$},\\
  h(w) &\text{otherwise,}
\end{cases}
\quad\text{and}\quad
\beta'(w)=
\begin{cases}
  \beta(w) &\text{if $w\in N_G(v)$},\\
  1 &\text{otherwise.}
\end{cases}
        \]
	Clearly, $\alpha' \times \beta'$ is a proper $(c_1+1)c_2$-coloring of $G\setminus v$.

	Thus we may assume $\abs{V(T)}>1$.
	Let $t_0$ be the unique node of $T$ such that $v \in V(\phi(t_0))$. Let $G_0:=\phi(t_0)$.
	Let $t_1$, $t_2$, $\ldots$, $t_m$ be the neighbors of $t_0$ in $T$.
        For each $i\in\{1,2,\ldots,m\}$,
        let $v_i\in V(G_0)$ and $u_i\in V(\phi(t_i))$ be vertices such that
        $\psi(t_0t_i)=\{v_i,u_i\}$
        and let $T_i$ be the connected component of $T\setminus t_0$ containing $t_i$.
        For each $i\in\{1,2,\ldots,m\}$,
        let $\phi_i$ be the restriction of $\phi$ on $V(T_i)$
        and $\psi_i$ be the restriction of $\psi$ on $E(T_i)$
        and $G_i$ be the graph constructed from a composition tree $(T_i,\phi_i,\psi_i)$.

        Let $h :V(G_0)\to \{1,2,\ldots,c_1\}$ be a proper  $c_1$-coloring of $G_0$.
	Let $\alpha'_0$, $\beta'_0$ be maps defined on $V(G_0)\setminus \{v,v_1,\ldots,v_m\}$ such that
        for $w\in V(G_0)\setminus \{v,v_1,\ldots,v_m\}$,
        \[
\alpha'_0(w)=
\begin{cases}
  0 &\text{if $w\in N_{G_0}(v)$},\\
  h(w) &\text{otherwise,}
\end{cases}
\quad\text{ and }\quad
\beta'_0(w)=
\begin{cases}
  \beta(w) &\text{if $w\in N_G(v)$},\\
  1 &\text{otherwise.}
\end{cases}
        \]

        Now we are going to define, for each $i\in\{1,2,\ldots,m\}$,  a proper $c_2$-coloring $\beta_i$ of $G_i[N_{G_i}(u_i)]$.
        If $v_i$ is adjacent to $v$ in $G_0$, then
        $N_{G_i}(u_i)$ is a subset of $N_G(v)$
        and so let us define $\beta_i$ to be 
        the proper $c_2$-coloring  of $G[N_{G_i}(u_i)]$
        induced by~$\beta$.

        If $v_i$ is non-adjacent to $v$ in $G_0$, then
        we claim that 
        there exists a vertex $y$ of $G$
        such that $N_{G_i}(u_i)\subseteq N_G(y)$.
        Since $G_0$ is connected,
        $v_i$ has a neighbor $x$ in $G_0$.
        If $x$ is not a marker vertex, then $N_{G_i}(u_i)\subseteq N_G(x)$.
        If $x$ is a marker vertex, say $x=v_j$ for some $j\neq i$, then
        there exists a neighbor $y$ of $u_j$ in $G_j$ because $G_j$ is connected and $\abs{V(G_j)}\neq 1$.
        Now we observe that $N_{G_i}(u_i)\subseteq N_G(y)$.
        This proves the claim.
        By the claim, we can define 
        $\beta_i$ as a proper $c_2$-coloring 
        of $G_i[N_{G_i}(u_i)]$ induced by 
        a proper $c_2$-coloring of $G[N_G(y)]$.

        Observe that
        $\abs{V(G_i)}<\abs{V(G)}$ for all $i\in\{1,2,\ldots,m\}$
        because $G_0$ has at least three vertices.
        Now, 
	by the induction hypothesis, for each $i\in \{1,2,\ldots,m\}$,
	there exist maps
        $\alpha'_i:V(G_i\setminus u_i )\to \{0,1,\ldots,c_1\}$ and $\beta_i':V(G_i\setminus u_i) \to \{1,2,\ldots,c_2\}$ 
	such that for every $w\in N_{G_i}(u_i)$,
        \begin{align*}
          \alpha_i'(w)&=
          \begin{cases}
            h(v_i)&\text{if  $v_i$ is non-adjacent to $v$ in $G_0$},\\
            0 & \text{otherwise,}
          \end{cases}\\
           \beta_i'(w)&=\beta_i(w), 
        \end{align*}
        and 
        $\alpha_i'\times \beta_i'$ is a proper $(c_1+1)c_2$-coloring of $G_i\setminus u_i$,
        because 
        we may swap colors $0$ and $h(v_i)$ in $\alpha'_i$
        after applying the induction hypothesis.

        Now we define maps $\alpha'$ and $\beta'$ on $V(G)\setminus\{v\}$ such that
        for $w\in V(G)\setminus \{v\}$, 
        \[
          \alpha'(w)=\alpha_i'(w)
          \text{ and }\beta'(w)=\beta'_i(w)
          \text{ if $ w\in V(G_i)$
                      for some $i\in\{0,1,2,\ldots,m\}$. }
          \]
	Clearly, $\beta'$ extends $\beta$. In addition, $\alpha'(w)=0$ for all neighbors $w$ of $v$ in $G$.

	We claim that $c=\alpha' \times \beta'$ is a proper coloring of $G\setminus v$.
	Let $x,y \in V(G\setminus v)$ be adjacent vertices in $G\setminus v$.
        If both $x$ and $y$ are neighbors of $v$,
        then $\beta'(x)=\beta(x)\neq \beta(y)=\beta'(y)$.
        So we may assume that $y$ is not a neighbor of $v$.
	\begin{itemize}
        \item If $x,y\in V(G_0)$, then 
          $\alpha'(x)\neq\alpha'(y)$  because $\alpha'(x)\in\{0,h(x)\}$ and  $\alpha'(y)=h(y)\neq0$.

	\item If $x,y \in V(G_i)$ for some $i\in\{1,2,\ldots,m\}$, then
          $(\alpha'(x),\beta'(x))\neq (\alpha'(y),\beta'(y))$
          because $\alpha_i'\times \beta_i'$ is a proper coloring of $G_i\setminus u_i$.
	\item If $V(G_0)$ contains exactly one of $x$ and $y$, say $x$, then
          there exists $i\in\{1,2,\ldots,m\}$ such that $y\in V(G_i)$.
          Then $x$ is adjacent to $v_i$ in $G_0$
          and $y$ is adjacent to $u_i$ in $G_i$.
          Since $y$ is not adjacent to $v$, 
          $v_i$ is not adjacent to $v$ in $G_0$
          and so $\alpha'(y)=\alpha_i'(y)=h(v_i)$. As $\alpha'(x)\in\{0,h(x)\}$,
          we deduce that $\alpha'(x)\neq \alpha'(y)$.

	\item If $x\in V(G_i)$ and $y \in V(G_j)$ for distinct $i,j\in\{1,2,\ldots,m\}$, then
          $x$ is adjacent to $u_i$ in $G_i$,
          $v_i$ is adjacent to $v_j$ in $G_0$,
          and $u_j$ is adjacent to $y$ in $G_j$.
          Since $y$ is not adjacent to $v$, 
          $v_j$ is not adjacent to $v$ in $G_0$
          and so $\alpha'(y)=\alpha_j'(y)=h(v_j)$. 
          Note that $\alpha'(x)\in\{0,h(v_i)\}$ and therefore
          $\alpha(x)\neq \alpha(y)$ because  $h$ is a proper coloring of $G_0$.
 	\end{itemize}
 Therefore, $c$ is a proper coloring of $G\setminus v$. This completes the proof.
	\end{proof}

\begin{proof}[Proof of Theorem~\ref{THM:main}]
  We may assume that $\mathcal G$ is hereditary,
  by replacing $\mathcal G$ with
  the closure of $\mathcal G$
  under isomorphism and taking induced subgraphs, if necessary.
  
  Let $f$ be a $\chi$-bounding function for $\mathcal{G}$ that is a polynomial.
  We may assume that $1\le f(0)\le f(1)\le f(2)\le \cdots$, by replacing $f(x)=\sum_{i} a_i x^i$ with
  $\sum_i \abs{a_i}x^i$ if needed.
  By Theorem~\ref{thm:substitution}, $\gsub$ is $\chi$-bounded
  by a polynomial $f^*$. We may assume that $1\le f^*(0)\le f^*(1)\le f^*(2)\le \cdots$.
  
  We claim that
  \[\chi(G) \le (f(\omega(G))+1)f^*(\omega(G)-1)\]
  for all $G\in \gjoin$.
  This claim implies the theorem because $\gjoin$ is hereditary by Lemma~\ref{lem:hereditary}.

Let $k=\omega(G)$.
We may assume that $k>1$.
We may assume that $G$ is connected because $\gjoin$ is hereditary
and both $f$ and $f^*$ are non-decreasing.
We may assume that $G$ has at least three vertices.
By Lemma~\ref{lem:composition},
$G$ has a composition tree $(T,\phi,\psi)$ with $\phi(x)\in \mathcal{G}$ for every node $x$ of $T$.
Note that $\omega(\phi(x))\le k$ because $\phi(x)$ is isomorphic to an induced subgraph of $G$
and therefore $\chi(\phi(x))\le f(\omega(\phi(x)))\le f(k)$.
For each vertex $w \in V(G)$,
$\omega(G[N_{G}(w)])\le k-1$ and 
$G[N_{G}(w)]$ belongs to $\gsub$ by Lemma~\ref{LEM:substitution}, and so $G[N_{G}(w)]$ is $f^*(k-1)$-colorable.
Let $v$ be a vertex of $G$.
By  Lemma~\ref{LEM:main},
there exist a proper $(f(k)+1)f^*(k-1)$-coloring  $c=\alpha'\times \beta'$
of $G\setminus v$ 
such that $\alpha'(w)=0$ for every neighbor $w$ of $v$.
Then we can easily extend this to a proper $(f(k)+1)f^*(k-1)$-coloring of $G$
by taking $\alpha'(v)\neq 0$.
\end{proof}
We remark that the same method can also prove that if $\mathcal G$
is a class of graphs having an exponential $\chi$-bounding function,
then so is $\gjoin$. This is because
Chudnovsky, Penev, Scott, and Trotignon~\cite{CPST2013}
also prove an analogue of Theorem~\ref{thm:substitution}
for classes of graphs having  exponential $\chi$-bounding functions.

\section{Graphs without
   $P_n$ or $K_{\lceil n/2\rceil} \tri \S_{\lceil n/2\rceil}$ induced subgraphs}
\label{sec:path}

Now we are ready to prove Theorem~\ref{thm:pathhalf},
which states
that
the class of graphs with no induced subgraph isomorphic to $P_n$
or $K_{\lceil n/2\rceil }\tri \S_{\lceil n/2\rceil}$ is polynomially $\chi$-bounded.

For this section,
a \emph{path} from $v$ to $w$ is a sequence $v_0v_1\cdots v_\ell$ of distinct vertices such
that
$v_0=v$, $v_\ell=w$, and 
$v_i$ is adjacent to $v_{i-1}$ for all $i=1,2,\ldots,\ell$.
We say a path $Q:=w_0w_1\cdots w_m$ \emph{extends} a path $P:=v_0v_1\cdots v_\ell$ if $m\ge \ell$ and $v_i=w_i$ for all $i\in\{0,1,\ldots,\ell\}$.

For an induced path $P$ from $v$ to $w$ in $G$,
we write $\Omega(G,P)$ to denote $G\setminus (V(P)\cup N_G(V(P\setminus w)))$.
A component of $\Omega(G,P)$ is \emph{attached to} $P$
if it contains a neighbor of $w$.
A component $C$ of $\Omega(G,P)$ is \emph{$d$-good}
if the neighbors of $w$ in $C$ induce a graph of chromatic
number larger than $d$.
We say $C$ is \emph{$d$-bad} if it is not $d$-good.
We say $P$ is \emph{$d$-good} in $G$ if $\Omega(G,P)$
has a $d$-good component.

\begin{LEM}\label{lem:start}
  Let $k$, $d$ be positive integers. 
  Let $G$ be a graph.
  If $\omega(G)\le k$ and  $\chi(G)>kd$,  then $G$ has 
  a path $P$ of length $1$
  such that $\Omega(G,P)$ has a component $C$
  attached to $P$
  and $\chi(C)>d$.
\end{LEM}
\begin{proof}
  We may assume that $G$ is connected.
  Let $K$ be a maximum clique of $G$.
  By assumption, $\abs{K}\le k$.
  For each vertex $x$ of $K$, let
  $H_x=G\setminus N_G(x)$.
  Since $K$ is a maximum clique, for every vertex $y$,
  there is $x\in K$ such that $y\in V(H_x)$
  and therefore $\chi(G)\le \sum_{x\in K}\chi(H_x)$.
  So there exists $x\in K$ such that
  $\chi(H_x)>d$.
  Let $C'$ be a component of $H_x$ such that $\chi(C')=\chi(H_x)$.
  Since $\chi(C')>1$ and $x$ is an isolated vertex in $H_x$,
  $x\notin V(C')$.
  Since $G$ is connected,
  there is a shortest path $v_0v_1\cdots v_\ell$ from $x=v_0$ to some vertex $v_\ell$ of $C'$.
  Let $v:=v_{\ell-2}$, $w:=v_{\ell-1}$, and $P:=vw$. Let $C$ be the component of $G\setminus N_G(v)$ containing $C'$.
  Then $C$ is a component attached to $P$ and $\chi(C)>d$.
\end{proof}

\begin{LEM}\label{lem:extend}
  If a graph $G$ has an induced path $P$ of length at least $1$ and
  $\Omega(G,P)$ has a $d$-bad component $C$ attached to $P$
  with $\chi(C)>d$,
  then
  there exist
  an induced path $P'$ extending $P$ by exactly $1$ edge
  and 
  a component $C'$ of $\Omega(G,P')$ attached to $P'$
  such that 
  \[\chi(C')\ge \chi(C)-d.    \]
\end{LEM}
\begin{proof}
  Let $w$ be the last vertex of $P$.
  Let $C_w$ be the subgraph of $C$ induced by the neighbors of $w$.
  Since $C$ is $d$-bad,
  $\chi(C_w)\le d$ and therefore
  $\chi(C\setminus N_G(w))\ge \chi(C)-
  \chi(C_w)\ge \chi(C)-d>0$.
  So $C\setminus N_G(w)$ has a component $C'$
  with $\chi(C')\ge \chi(C)-d$.
  Since $C$ is connected,  there is a vertex $w'\in V(C_w)$
  adjacent to some vertex in $C'$.
  We obtain $P'$ by adding $w'$ as a last vertex to $P$. Then $C'$
  is a component of $\Omega(G,P')$ attached to $P'$.
\end{proof}
\begin{LEM}\label{lem:dompath}
  Let $n\ge 4$.
  Let $G$ be a graph having no induced subgraph isomorphic to $P_n$.
  Let $P$ be a path of length $1$.
  If $\Omega(G,P)$ has a component $C$ attached to $P$ with $\chi(C)>d(n-3)$,
  then $G$ has a $d$-good induced path $P'$ extending $P$.
\end{LEM}
\begin{proof}
  Suppose that $G$ has no $d$-good induced path extending $P$.
  By applying Lemma~\ref{lem:extend} $(n-3)$ times,
  we can find an induced path $P'$ of length $n-2$ extending $P$
  and a component $C'$ of $\Omega(G,P')$ attached to $P'$
  such that $\chi(C')\ge \chi(C)-  d(n-3)>0$.
  We obtain an induced path of length $n-1$
  by taking $P'$ and one vertex in $C'$ adjacent to the last vertex of $P'$. This contradicts the assumption that $G$ has
  no induced path on $n$ vertices.
\end{proof}

Now we are ready to prove Theorem~\ref{thm:pathhalf}. 
\begin{THMMAIN1}%
  Let $n\ge 4$.
  If a graph $G$ has no induced subgraph isomorphic to 
  $P_n$ or $K_{\lceil n/2\rceil}\tri \S_{\lceil n/2\rceil }$,
  then 
 \[\chi(G)\le (n-3)^{\lceil n/2\rceil-1} \omega(G)^{\lceil n/2\rceil-1}.\]
\end{THMMAIN1}
\begin{proof}
  Let $k=\omega(G)$.
  We may assume that $G$ is connected.
  Let $G_0=G$.
  Let $d_i= (n-3)^{\lceil n/2\rceil-i-1} k^{\lceil n/2\rceil-i-1}$.
  Note that $\chi(G_0)>d_0$.
  
  Inductively we will find, in $G_{i-1}$ of $\chi(G_{i-1})>d_{i-1}$,  an induced path $Q_i$
  and connected induced subgraphs $C_i$,  $G_i$ of $\chi(G_i)>d_i$ as follows.
  For $i=1,\ldots, \lceil n/2\rceil-1 $,
  by Lemmas~\ref{lem:start} and \ref{lem:dompath},
  $G_{i-1}$ has a $d_i$-good induced path $Q_{i}$
  of length at least $1$, because $d_{i-1}=d_i k(n-3)$.
  Let $C_i$ be a $d_i$-good component of $\Omega(G_{i-1},Q_i)$ attached to $Q_i$.
  Among all components of the subgraph of $C_i$ induced by the neighbors of the last vertex of $Q_i$,
  we choose a component $G_i$ of the maximum chromatic number.
  By the definition of a $d_i$-good component, $\chi(G_i)>d_i$.
  This constructs $G_1$, $G_2$, $\ldots$, $G_{\lceil n/2\rceil -1}$.

  As $\chi(G_{\lceil n/2\rceil -1})>d_{\lceil n/2\rceil-1}=1$,
  $G_{\lceil n/2\rceil -1}$ contains at least one edge $xy$.
  By collecting the last two vertices of
  $Q_1$, $Q_2$, $\ldots$, $Q_{\lceil n/2\rceil-1}$ and $x$, $y$,
  we obtain an induced subgraph
  isomorphic to $K_{\lceil n/2\rceil }\tri \S_{\lceil n/2\rceil}$,
  contradicting the assumption on $G$.
\end{proof}
	
\section{Discussions}\label{sec:p5}
Corollary~\ref{cor:path} and Theorem~\ref{thm:pathhalf}
show that 
\begin{equation}\label{eq:bound}
  \chi(G)\le (n-3)^{\lceil n/2\rceil-1} \omega(G)^{\lceil n/2\rceil-1}.
\end{equation}
for a graph $G$ with no vertex-minor isomorphic to $P_n$.
It is natural to ask the following question.
\begin{QUE}\label{q}
  Do there exist a constant $c$ and a function $f(n)$ such that
  \[\chi(G)\le f(n) \omega(G)^c\]
  for all integers $n$ and all graphs $G$ with no vertex-minor isomorphic to $P_n$?
\end{QUE}

If we replace the vertex-minor condition of $G$ in Question~\ref{q}
with the condition having no induced subgraph
isomorphic to $P_n$ or $K_{\lceil n/2\rceil}\tri \S_{\lceil n/2\rceil}$,
then the answer is no, by an argument in \cite{surveychi}.
We include its proof for the completeness.
\begin{PROP}\label{prop:q}
  For every constant $c$ and a function $f(n)$,
  there exist  a graph $G$ and an integer $n$ such that
  $\chi(G)>f(n)\omega(G)^c$
  and $G$ has no induced subgraph isomorphic to $P_n$ or $K_{\lceil n/2\rceil}\tri \S_{\lceil n/2\rceil}$.
\end{PROP}
\begin{proof}
  There are known lower bounds on Ramsey numbers; for instance, Spencer~\cite{Spencer1975} showed that
  for fixed $k$, $R(k,t)>t^{(k-1)/2+o(1)}$ for all sufficiently large $t$.
  It implies that for fixed $k$, for all sufficiently large $t$,
  there exists a graph $G^{(k)}_{t}$ having no stable set of size $k$
  and $\omega(G^{(k)}_t)<t$, $\abs{V(G^{(k)}_t)}>t^{(k-1)/2+o(1)}$.
  Then $\chi(G^{(k)}_t)\ge t^{(k-1)/2+o(1)}/k $.

  Fix an integer $k$ so that $(k-1)/2>c+1$.
  Let $c':=f(2k)$.
  Then for all sufficiently large $t$,
  $\chi(G^{(k)}_t)> t^{(k-1)/2+o(1)}/k \ge c' t^c$.
  Then for all sufficiently large $t$,
  $\chi(G^{(k)}_t)>c' \omega(G)^c$ and 
  $G^{(k)}_t$ has no induced subgraph isomorphic to $P_{2k}$
  or $K_{k}\tri \S_k$.
\end{proof}

The bound \eqref{eq:bound} cannot be improved for $n=4$, as $\chi(G)\ge \omega(G)$ in general.
We will present the best possible bound for $n=5$.
\begin{THM}\label{THM:p5free}
  If a graph $G$ has no vertex-minor isomorphic to $P_5$,
  then
  \[\chi(G)\le \omega(G)+1.\]
\end{THM}

The following proposition trivially implies Theorem~\ref{THM:p5free}.
We denote by $W_n$ the wheel graph on $n+1$ vertices.

\begin{PROP}\label{PROP:P5free}
  Every graph with no vertex-minor isomorphic to $P_5$
  is perfect,
  unless it has a component isomorphic to $C_5$ or $W_5$.
\end{PROP}

In order to prove Proposition~\ref{PROP:P5free}, 
we need to define the following graph classes. See Figure~\ref{FIG:p5free} for an illustration.
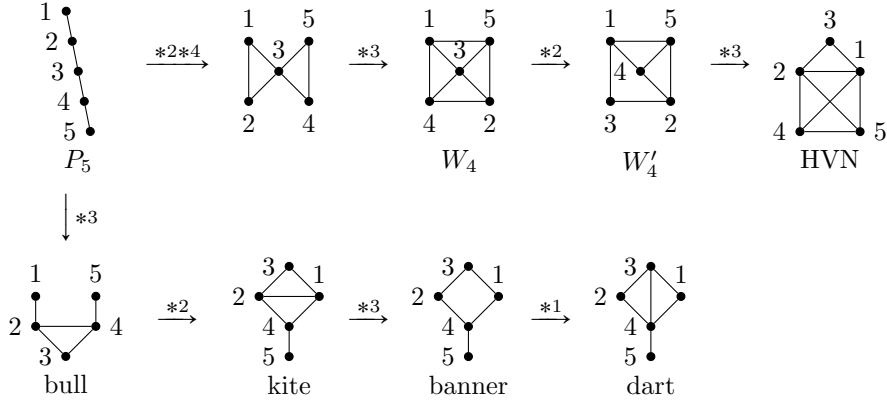
\begin{figure}
  \tikzstyle{v}=[circle, draw, solid, fill=black, inner sep=0pt, minimum width=3pt]
  \[
    \minCDarrowwidth0.5em
    \begin{CD}
    \begin{tikzpicture}[baseline=(v3),scale=.8]
      \foreach \i in {1,2,3,4,5}
      {
        \node [v,label=left:$\i$] (v\i) at (0.1*\i,-0.5*\i) {};
      }
      \draw (v1)--(v2)--(v3)--(v4)--(v5);
      \node [label=below:$P_5$] at(0.3,-2.5){};
    \end{tikzpicture}
    @>*2*4>>
        \begin{tikzpicture}[baseline=(v3),scale=.8]
          \node [v,label=$1$] (v1) at (0,1) {};
          \node [v,label=below:$2$] (v2) at (0,0) {};
          \node [v,label=$3$] (v3) at (0.5,0.5) {};
          \node [v,label=below:$4$] (v4) at (1,0) {};
          \node [v,label=$5$] (v5) at (1,1) {};
          \draw (v1)--(v2)--(v3)--(v4)--(v5);
          \draw (v1)--(v3);
          \draw (v3)--(v5);
        \end{tikzpicture}
        @>*3>>
        \begin{tikzpicture}[baseline=(v3),scale=.8]
          \node [v,label=$1$] (v1) at (0,1) {};
          \node [v,label=below:$2$] (v2) at (1,0) {};
          \node [v,label=$3$] (v3) at (0.5,0.5) {};
          \node [v,label=below:$4$] (v4) at (0,0) {};
          \node [v,label=$5$] (v5) at (1,1) {};
          \draw (v1)--(v5)--(v2)--(v4)--(v1);
          \foreach \i in {1,2,4,5}
          {\draw (v3)--(v\i);}
          \node [label=below:$W_4$] at (0.5,-.5){};
        \end{tikzpicture}
        @>*2>>
        \begin{tikzpicture}[baseline=(v4),scale=.8]
          \node [v,label=$1$] (v1) at (0,1) {};
          \node [v,label=below:$2$] (v2) at (1,0) {};
          \node [v,label=below:$3$] (v3) at (0,0) {};
          \node [v,label=left:$4$] (v4) at (0.5,0.5) {};
          \node [v,label=$5$] (v5) at (1,1) {};
          \draw (v1)--(v5)--(v2)--(v3)--(v1);
          \foreach \i in {1,2,5}
          {\draw (v4)--(v\i);}
          \node [label=below:$W_4'$] at (0.5,-.5){};
        \end{tikzpicture}
        @>*3>>
        \begin{tikzpicture}[baseline=(v2),scale=.8]
          \node [v,label=$1$] (v1) at (1,1){};
          \node [v,label=left:$2$] (v2) at (0,1){};
          \node [v,label=$3$] (v3) at  (0.5,1.5){};
          \node [v,label=left:$4$] (v4) at (0,0){};
          \node [v,label=right:$5$] (v5) at (1,0) {};
          \draw (v1)--(v5)--(v2)--(v3)--(v1);
          \foreach \i in {1,2,5}
          {\draw (v4)--(v\i);}
          \node [label=below:HVN] at (0.5,0){};
          \draw (v1)--(v2);
        \end{tikzpicture}
        \\
        @VV{*3}V\\
        \begin{tikzpicture}[baseline=(v2),scale=.8]
          \node [v,label=$1$] (v1) at (0,.5) {};
          \node [v,label=left:$2$] (v2) at (0,0) {};
          \node [v,label=left:$3$] (v3) at (0.5,-0.5) {};
          \node [v,label=right:$4$] (v4) at (1,0) {};
          \node [v,label=$5$] (v5) at (1,.5) {};
          \draw (v1)--(v2)--(v3)--(v4)--(v5);
          \draw(v2)--(v4);
          \node [label=below:bull] at (0.5,-.5){};
        \end{tikzpicture}
        @>*2>>
        \begin{tikzpicture}[baseline=(v4),scale=.8]
          \node [v,label=$1$] (v1) at (.5,0) {};
          \node [v,label=left:$2$] (v2) at (-.5,0) {};
          \node [v,label=left:$3$] (v3) at (0,.5) {};
          \node [v,label=left:$4$] (v4) at (0,-.5) {};
          \node [v,label=left:$5$] (v5) at (0,-1) {};
          \draw (v1)--(v3)--(v2)--(v4)--(v5);
          \draw (v4)--(v1)--(v2);
          \node [label=below:kite] at (0,-1){};
        \end{tikzpicture}
      @>*3>>
        \begin{tikzpicture}[baseline=(v4),scale=.8]
          \node [v,label=$1$] (v1) at (.5,0) {};
          \node [v,label=left:$2$] (v2) at (-.5,0) {};
          \node [v,label=left:$3$] (v3) at (0,.5) {};
          \node [v,label=left:$4$] (v4) at (0,-.5) {};
          \node [v,label=left:$5$] (v5) at (0,-1) {};
          \draw (v1)--(v3)--(v2)--(v4)--(v5);
          \draw (v4)--(v1);
          \node [label=below:banner] at (0,-1){};
        \end{tikzpicture}
      @>*1>>
      \begin{tikzpicture}[baseline=(v4),scale=.8]
          \node [v,label=$1$] (v1) at (.5,0) {};
          \node [v,label=left:$2$] (v2) at (-.5,0) {};
          \node [v,label=left:$3$] (v3) at (0,.5) {};
          \node [v,label=left:$4$] (v4) at (0,-.5) {};
          \node [v,label=left:$5$] (v5) at (0,-1) {};
          \draw (v1)--(v3)--(v2)--(v4)--(v5);
          \draw (v3)--(v4)--(v1);
          \node [label=below:dart] at (0,-1){};
        \end{tikzpicture}
      \end{CD}
  \]
  \caption{Some graphs locally equivalent to $P_5$.}
  \label{FIG:p5free}
\end{figure}

\begin{itemize}
\item $W_4'$: the graph obtained from $W_4$ by deleting a spoke.
\item Banner: the graph obtained from $C_4$ by adding a pendant edge.
\item Bull: the graph obtained from $C_3$ by adding two pendant edges to distinct vertices of $C_3$.
\item Dart: the graph obtained from $K_4\setminus e$ for some edge $e$ of $K_4$ by adding a pendant edge to a vertex of degree $3$.
\item HVN: the graph obtained from $K_4$ by adding a vertex of degree $2$.
\item Kite: the graph obtained from $K_4\setminus e$ for some edge $e$ of $K_4$ by adding a pendant edge to a divalent vertex.
\end{itemize}

We say that $G$ is \emph{$H$-free}
if $G$ has no induced subgraph isomorphic to $H$.
We write $\overline{G}$ to denote the \emph{complement} of a graph $G$.
\begin{proof}[Proof of Proposition~\ref{PROP:P5free}]
  From Figure~\ref{FIG:p5free}, it is easy to check that
  $W_4$, $W_4'$, a banner, a bull, a dart, an HVN and a kite are locally equivalent to $P_5$.
  Therefore, $G$ has no induced subgraph isomorphic
  to any of those graphs.

  We may assume that $G$ is connected.
  If $G$ is $C_5$-free, 
  then $G$ does not contain an odd hole because $G$ is $P_5$-free.
  Since $\overline{C_5}$ is isomorphic to $C_5$, $G$ is $\overline{C_5}$-free. 
  In addition, since $\overline{W_4'}$ is the disjoint union of $P_2$ and $P_3$, $G$ is $\overline{C_k}$-free for every odd $k\ge 7$.
  Therefore $G$ is perfect by the strong perfect graph theorem~\cite{CRST2006}.
  
  Now we may assume that 
  $G$ contains $C_5$ as an induced subgraph.
  Let $L_i$ be the set of vertices of $G$ having the distance $i$ to $C_5$.
We may assume that $G$ is not $C_5$, that is, $L_1$ is not empty.

We claim that $L_1$ is complete to $L_0$.
Suppose $v \in L_1$ is not complete to $L_0$. Then $v$ has exactly 1, 2, 3 or 4 neighbors in $L_0$. In each case it is easy to check that we can find an induced subgraph isomorphic to $P_5$, a bull, a banner or a kite, a contradiction. 

Now we claim that $L_2=\emptyset$.
Suppose $v \in L_2$. Let $u \in L_1$ such that $uv$ is an edge. Now we see that $G$ contains a dart, a contradiction.

If two vertices $u$, $v$ in $L_1$ are adjacent, then $G$ contains a HVN as an induced subgraph, a contradiction.  Thus, $L_1$ is stable.

If $L_1$ contains more than one vertex, then $G$ contains $W_4$, a contradiction. So $\abs{L_1}=1$, and so $G=W_5$.
\end{proof}

Finally let us discuss pivot-minors instead of vertex-minors. 
We may ask the following question, which is stronger than Question~\ref{q:polygeelen}.
\begin{QUE}\label{q:pivot}
  Is it true that for every graph $H$,
  the class of graphs with no pivot-minor isomorphic to $H$ is polynomially
  $\chi$-bounded?
\end{QUE}
We will show that this is true if $H$ is a path.
Here is a useful lemma, replacing Lemma~\ref{lem:ko}.
\begin{LEM}
  The graph $K_n\tri \S_n$ has a pivot-minor isomorphic to $P_{n+1}$.
\end{LEM}
\begin{proof}
  Let us remind that $\{a_1,a_2,\ldots,a_n\}$ is a clique and $\{b_1,b_2,\ldots,b_n\}$ is an independent set and $a_i$ is adjacent to $b_j$ if and only if $i\ge j$.
  It is easy to observe that 
  $(K_n\tri\S_n)\pivot a_2b_2\pivot a_3b_3\pivot a_4b_4 \pivot\cdots\pivot a_{n-1}b_{n-1}$
  has an induced path $a_1a_2a_3\cdots a_nb_n$ of length $n$.
\end{proof}
So we deduce the following strengthening of Corollary~\ref{cor:path} by the same proof.
\begin{PROP}
  The class of graphs with no pivot-minor isomorphic to $P_n$
  is polynomially $\chi$-bounded.
\end{PROP}
However, the analogue of Theorem~\ref{THM:pathtocycle} is false for pivot-minors.
As mentioned in \cite{CKO2015},
if $k\not\equiv n\pmod 2$, then 
$C_k$ has no pivot-minor isomorphic to $C_n$.
This implies that there is a prime graph with an arbitrary long induced path
having no pivot-minors isomorphic to $C_n$.

\paragraph{Acknowledgement.}
The authors would like to thank the anonymous reviewer for suggesting
various improvements and in particular,
Question~\ref{q} and Proposition~\ref{prop:q}.
\providecommand{\bysame}{\leavevmode\hbox to3em{\hrulefill}\thinspace}
\providecommand{\MR}{\relax\ifhmode\unskip\space\fi MR }
\providecommand{\MRhref}[2]{%
  \href{http://www.ams.org/mathscinet-getitem?mr=#1}{#2}
}
\providecommand{\href}[2]{#2}

\end{document}